\documentclass[a4paper]{amsart}
\usepackage{amsmath}
\usepackage{amsfonts}

\newtheorem{theorem}{Theorem}
\theoremstyle{plain}

\newtheorem{corollary}{Corollary}

\newtheorem{definition}{Definition}

\newtheorem{lemma}{Lemma}

\newtheorem{remark}{Remark}

\input{tcilatex}

\begin{document}
\title[Ambartsumian Type Theorems]{Ambarzumyan Type Theorems on a Time Scale}
\author{A. Sinan Ozkan}
\curraddr{Department of Mathematics, Faculty of Science, Cumhuriyet
University 58140 \\
Sivas, TURKEY}
\email{sozkan@cumhuriyet.edu.tr}
\subjclass[2000]{ 31B20, 39A12, 34B24}
\keywords{Ambarzumyan Theorem, Time scale, Sturm-Liouville equation; inverse
problem; dynamic equations.}

\begin{abstract}
In this paper, we consider a Sturm--Liouville dynamic equation with Robin
boundary conditions on time scale and investigate the conditions which
{}{}guarantee that the potential function is specified.
\end{abstract}

\maketitle

\section{\protect\bigskip \textbf{Introduction}}

Time scale theory was introduced by Hilger in order to unify continuous and
discrete analysis \cite{Hilger}. From then on this approach has received a
lot of attention and has applied quickly to various area in mathematics.
Sturm--Liouville theory on time scales was studied first by Erbe and Hilger 
\cite{Erbe} in 1993. Some important results on the properties of eigenvalues
and eigenfunctions of a Sturm--Liouville problem on time scales were given
in various publications (see e.g. \cite{Agarwal}, \cite{Amster}, \cite%
{Amster2}, \cite{Davidson}-\cite{Davidson3}, \cite{Erbe2}, \cite{Guseinov}-%
\cite{Sun} and the references therein).

Inverse spectral problems consist in recovering the coefficients of an
operator from their spectral characteristics. Althouhgh there are vast
literature for inverse Sturm--Liouville problems on a continuous interval,
there are no study on the general time scales. For Sturm-Liouville operator
on a continuous interval, the study which starts inverse spectral theory,
was published by Ambarzumyan \cite{Ambartsumian} in 1929. He prove that: if $%
q$ is continuous function on $(0,1)$ and the eigenvalues of the problem

\begin{eqnarray*}
&&\left. -y^{\prime \prime }+q(t)y=\lambda y,\text{ }t\in (0,1)\right.  \\
&&\left. y^{\prime }(0)=y^{\prime }(1)=0\right. 
\end{eqnarray*}%
are given as $\lambda _{n}=n^{2}\pi ^{2},$ $n\geq 0$ then $q\equiv 0.$ 

Freiling and Yurko \cite{Freiling} generalized this result as $\lambda
_{0}=\int\limits_{0}^{1}q(t)dt$ implies $q\equiv \lambda _{0}.$

The goal of this paper to prove an Ambarzumyan type theorem on a general
time scale and to apply it on the a special time scale. \ In our main
result, Theorem 1, we generalize the results of Freiling and Yurko for
Sturm-Liouville operator with more general boundary conditions on a time
scale. 

\section{\textbf{Preliminaries and Main Results}}

If $\mathbb{T}$ is a closed subset of $%
%TCIMACRO{\U{211d} }%
%BeginExpansion
\mathbb{R}
%EndExpansion
$ it called as a time scale. The jump operators $\sigma $, $\rho $ and
graininess operator on $\mathbb{T}$ are defined as follow: 
\begin{eqnarray*}
&&\left. \sigma :\mathbb{T}\rightarrow \mathbb{T},\text{ }\sigma \left(
t\right) =\inf \left\{ s\in \mathbb{T}:s>t\right\} \text{ if }t\neq \sup 
\mathbb{T}\text{,}\right.  \\
&&\left. \rho :\mathbb{T}\rightarrow \mathbb{T},\text{ }\rho \left( t\right)
=\sup \left\{ s\in \mathbb{T}:s<t\right\} \text{ if }t\neq \inf \mathbb{T}%
\text{,}\right.  \\
&&\left. \sigma \left( \sup \mathbb{T}\right) =\sup \mathbb{T},\text{ }\rho
\left( \inf \mathbb{T}\right) =\inf \mathbb{T}\text{,}\right.  \\
&&\left. \mu :\mathbb{T}\rightarrow \lbrack 0,\infty )\text{ }\mu \left(
t\right) =\sigma \left( t\right) -t\text{.}\right. 
\end{eqnarray*}

A point of $\mathbb{T}$ is called as left-dense, left-scattered,
right-dense, right-scattered and isolated if $\rho (t)=t$, $\rho (t)<t$, $%
\sigma (t)=t$, $\sigma (t)>t$ and $\rho (t)<t<\sigma (t),$ respectively.

A function $f:\mathbb{T}\rightarrow 
%TCIMACRO{\U{211d} }%
%BeginExpansion
\mathbb{R}
%EndExpansion
$ is called rd-continuous on $\mathbb{T}$ if it is continuous at all
right-dense points and has left-sided limits at all left-dense points in $%
\mathbb{T}$. The set of rd-continuous functions on $\mathbb{T}$ is denoted
by $C_{rd}(\mathbb{T)}$ or $C_{rd}$.

Put $\mathbb{T}^{k}:=\left\{ 
\begin{array}{cc}
\mathbb{T}-\{\sup \mathbb{T}\}\text{,} & \sup \mathbb{T}\text{ is
left-scattered} \\ 
\mathbb{T}\text{,} & \text{the other cases}%
\end{array}%
\right. ,$ $\mathbb{T}^{k^{2}}:=\left( \mathbb{T}^{k}\right) ^{k}.$

Let $t\in \mathbb{T}^{k}.$ Suppose that for given any $\varepsilon >0,$
there exist a neighborhood $U=\left( t-\delta ,t+\delta \right) \cap \mathbb{%
T}$ such that 
\begin{equation*}
\left\vert \left[ f\left( \sigma \left( t\right) \right) -f\left( s\right) %
\right] -f^{\Delta }\left( t\right) \left[ \sigma \left( t\right) -s\right]
\right\vert \leq \varepsilon \left\vert \sigma \left( t\right) -s\right\vert
\end{equation*}%
for all $s\in U$ then, $f$ is called differentiable at $t\in \mathbb{T}^{k}$%
. We call $f^{\Delta }\left( t\right) $ the delta derivative of $f$ at $t.$
A function $F:\mathbb{T}\rightarrow 
%TCIMACRO{\U{211d} }%
%BeginExpansion
\mathbb{R}
%EndExpansion
$ defined as $F^{\Delta }(t)=f(t)$ for all $t\in \mathbb{T}^{k}$ is called
an antiderivative of $f$ on $\mathbb{T}$. In this case the Cauchy integral
of $f$ is defined by

\begin{equation*}
\int\limits_{a}^{b}f\left( t\right) \triangle t=F\left( b\right) -F\left(
a\right) ,\text{ for }a,b\in \mathbb{T}\text{.}
\end{equation*}%
Some important relations whose proofs appear in \cite{Bohner}, chapter1 will
be needed. We collect them in the following lemma.

\begin{lemma}
Let $f:\mathbb{T}\rightarrow 
%TCIMACRO{\U{211d} }%
%BeginExpansion
\mathbb{R}
%EndExpansion
$, $g:\mathbb{T}\rightarrow 
%TCIMACRO{\U{211d} }%
%BeginExpansion
\mathbb{R}
%EndExpansion
$ be two functions and $t\in \mathbb{T}^{k}.$\newline
i) lf $f^{\Delta }\left( t\right) $ exists, then $f$ is continuous at $t$;%
\newline
ii) if $t$ is right-scattered and $f$ is continuous at $t$, then $f$ is
differentiable at $t$ and $f^{\Delta }\left( t\right) =\dfrac{f^{\sigma
}\left( t\right) -f\left( t\right) }{\sigma \left( t\right) -t},$ where $%
f^{\sigma }\left( t\right) =f\left( \sigma \left( t\right) \right) ;$\newline
iii) if $f^{\Delta }\left( t\right) $ exists, then $f^{\sigma }\left(
t\right) =f\left( t\right) +\mu (t)f^{\Delta }\left( t\right) ;$\newline
iv) if $f^{\Delta }\left( t\right) $,$g^{\Delta }\left( t\right) $ exist and 
$(fg)(t)$ is defined, then $(fg)^{\Delta }(t)=\left( f^{\Delta }g+f^{\sigma
}g^{\Delta }\right) (t)$ and if $\left( gg^{\sigma }\right) (t)\neq 0$, then 
$\left( \dfrac{f}{g}\right) ^{\Delta }(t)=\left( \dfrac{f^{\Delta
}g-fg^{\Delta }}{gg^{\sigma }}\right) (t)$;\newline
v) if $f\in C_{rd}(\mathbb{T)}$, then it has an antiderivative on $\mathbb{T}
$;\newline
vi) if $\mathbb{T}$ consists of only isolated points and $a,b\in \mathbb{T}$
with $a<b,$ then $\int\limits_{a}^{b}f\left( t\right) \triangle t=\sum_{t\in %
\left[ a,b\right) \cap \mathbb{T}}\mu (t)f(t);$\newline
vii) if $f(t)\geq 0$ for all $t\in \left[ a,b\right] \cap \mathbb{T}$ and $%
\int\limits_{a}^{b}f\left( t\right) \triangle t=0,$ then $f(t)\equiv 0.$
\end{lemma}

Throughout this paper we assume that $\mathbb{T}$ is a bounded time scale, $%
a=\inf \mathbb{T}$ and $b=\sup \mathbb{T}$. Consider the boundary value
problem $L=L\left( q,h_{a},h_{b}\right) $ generated by the Sturm--Liouville
dynamic equation%
\begin{equation}
\ell y:=-y^{\Delta \Delta }(t)+q(t)y^{\sigma }(t)=\lambda y^{\sigma }(t),%
\text{ \ \ }t\in \mathbb{T}^{k^{2}}
\end{equation}%
subject to the boundary conditions%
\begin{eqnarray}
&&\left. y^{\Delta }(a)-h_{a}y(a)=0\right. \medskip \\
&&\left. y^{\Delta }(\rho (b))-h_{b}y(\rho (b))=0\right.
\end{eqnarray}%
where $q(t)$ is real valued continiuous function on $\mathbb{T}$, $%
h_{a},h_{b}\in 
%TCIMACRO{\U{211d} }%
%BeginExpansion
\mathbb{R}
%EndExpansion
$ and $\lambda $ is the spectral parameter. Additionally, we assume that $%
a\neq $ $\rho (b)$, $1+h_{a}\mu (a)\neq 0$ and $1+h_{b}\mu (\rho (b))\neq 0.$

\begin{definition}
The values of the parameter for which the equation (1) has nonzero solutions
satisfy (2) and (3), are called eigenvalues and the corresponding nontrivial
solutions are called eigenfunctions.
\end{definition}

It is proven in \cite{Bohner} that all eigenvalues of the problem (1)-(3)
are real numbers.

\begin{definition}
A solution $y$ of (1) is said to have a zero at $t\in \mathbb{T}$ if $y(t)=0$%
, and it has a node between $t$ and $\sigma (t)$ if $y(t)y(\sigma (t))<0$. A
generalized zero of $y$ is then defined as a zero or a node.
\end{definition}

\begin{lemma}[\protect\cite{Agarwal}]
The eigenvalues of (1)-(3) may be arranged as $-\infty <\lambda _{1}<\lambda
_{2}<\lambda _{3}<...$ and an eigenfunction corresponding to $\lambda _{k+1}$
has exactly $k$ generalized zeros in the open interval $(a,b)$.
\end{lemma}

\begin{lemma}
If $y(t)$ is an eigenfunction of the problem (1)-(3) then $y^{\sigma
}(a)\neq 0$ and $y^{\sigma }(\rho (b))\neq 0.$
\end{lemma}

\begin{proof}
It is clear from Lemma 1 that $y^{\sigma }(a)=y(a)+\mu (a)y^{\Delta }(a)=y(a)%
\left[ 1+h_{a}\mu (a)\right] $ and $y^{\sigma }(\rho (b))=y(\rho (b))+\mu
(\rho (b))y^{\Delta }(\rho (b))=y(\rho (b))\left[ 1+h_{b}\mu (\rho (b))%
\right] $. We claim that $y(a)\neq 0$ and $y(\rho (b))\neq 0$. Otherwise,
from (2) and (3) $y^{\Delta }(a)=0$ or $y^{\Delta }(\rho (b))=0$ hold, then
by the uniqueness theorem of the solution of initial value problems $y(t)$
is identically vanish which contradicts that it is the eigenfunction.
Therefore the proof is completed from the assumption $1+h_{a}\mu (a)\neq 0,$ 
$1+h_{a}\mu (\rho (b))\neq 0.$
\end{proof}

\begin{theorem}
Let $\lambda _{1}$ be the first eigenvalue of (1)-(3). If 
\begin{equation*}
\lambda _{1}\geq \frac{1}{\rho (b)-a}\{h_{a}-h_{b}+\int\limits_{a}^{\rho
(b)}q(t)\Delta t\},
\end{equation*}%
then $q(t)\equiv \lambda _{1}.$
\end{theorem}

\begin{proof}
Let $y_{1}(t)$ be the corresponding eigenfunction\ to $\lambda _{1}.$ From
eq(1) and Lemma 2 we can write on $\mathbb{T}^{k^{2}}$

\begin{equation*}
\dfrac{y_{1}^{\Delta \Delta }(t)}{y_{1}^{\sigma }(t)}=q(t)-\lambda _{1}.
\end{equation*}%
It is from the relation%
\begin{equation*}
\dfrac{y_{1}^{\Delta \Delta }(t)}{y_{1}^{\sigma }(t)}=\dfrac{\left[
y_{1}^{\Delta }(t)\right] ^{2}}{y_{1}^{\sigma }(t)y_{1}(t)}+\left[ \dfrac{%
y_{1}^{\Delta }(t)}{y_{1}(t)}\right] ^{\Delta }
\end{equation*}%
that 
\begin{equation*}
\left[ \dfrac{y_{1}^{\Delta }(t)}{y_{1}(t)}\right] ^{\Delta }=q(t)-\lambda
_{1}-\dfrac{\left[ y_{1}^{\Delta }(t)\right] ^{2}}{y_{1}^{\sigma }(t)y_{1}(t)%
}.
\end{equation*}%
\newline
From Lemma 3 we can integration of both sides from $a$ to $\rho (b)$.
Therefore the following equality is obtained 
\begin{eqnarray*}
\int\limits_{a}^{\rho (b)}\dfrac{\left[ y_{1}^{\Delta }(t)\right] ^{2}}{%
y_{1}^{\sigma }(t)y_{1}(t)}\Delta t &=&\dfrac{y_{1}^{\Delta }(a)}{y_{1}(a)}-%
\dfrac{y_{1}^{\Delta }(\rho (b))}{y_{1}(\rho (b))}+\int\limits_{a}^{\rho (b)}%
\left[ q(t)-\lambda _{1}\right] \Delta t \\
&=&h_{a}-h_{b}+\int\limits_{a}^{\rho (b)}q(t)\Delta t-\lambda _{1}\left(
\rho (b)-a\right) .
\end{eqnarray*}%
It can be seen from Lemma 2 and our hypothesis that the right side of the
last equality is negative and the left side is non-negative. Thus $%
y_{1}^{\Delta }(t)\equiv 0$ and so $y_{1}(t)$ is constant. Substituting $%
y_{1}(t)$ is constant into equation (1), it is concluded that $q(t)\equiv
\lambda _{1}$.
\end{proof}

\begin{corollary}
The first eigenvalue of the problem $-y^{\Delta \Delta }+q(t)y^{\sigma
}=\lambda y^{\sigma },$ $y^{\Delta }(a)=y^{\Delta }(\rho (b))=0$ is $\lambda
_{1}=\frac{1}{\rho (b)-a}\int\limits_{a}^{\rho (b)}q(t)\Delta t$ then, $%
q(t)\equiv \lambda _{1}.$
\end{corollary}

This corollary is a generalization of the results of Freiling and Yurko \cite%
{Freiling} onto the time scale.

\begin{corollary}
Under the hypothesis $\int\limits_{a}^{\rho (b)}q(t)\Delta t=0$; if $%
q(t)\neq 0$, then the problem%
\begin{eqnarray*}
&&\left. -y^{\Delta \Delta }+q(t)y^{\sigma }=\lambda y^{\sigma },\right. \\
&&\left. y^{\Delta }(a)=y^{\Delta }(\rho (b))=0\right.
\end{eqnarray*}%
has at least one negative eigenvalue.
\end{corollary}

We conclude this paper with specializing our first result for a particular
time scale which consists of only isolated points.

\begin{remark}
Consider the time scale 
\begin{equation*}
\mathbb{T}=\{x_{k}\in 
%TCIMACRO{\U{211d} }%
%BeginExpansion
\mathbb{R}
%EndExpansion
:a=x_{0}<x_{1}<x_{2}<...<x_{n}=b\}
\end{equation*}%
and the following problem%
\begin{eqnarray*}
&&\left. -y^{\Delta \Delta }+q(t)y^{\sigma }=\lambda y^{\sigma },\right. \\
&&\left. y^{\Delta }(a)=y^{\Delta }(\rho (b))=0.\right.
\end{eqnarray*}%
If the first eigenvalue of the problem satisfy $\lambda _{1}\geq \max
\{q(t):t\in \mathbb{T}\},$ then $q(t)\equiv \lambda _{1}.$
\end{remark}

\end{document}